\documentclass[10pt]{article}
\usepackage{amsmath,fullpage,amsthm}
\usepackage{amscd}
\usepackage{amsthm,mathtools} \usepackage{amssymb}
\usepackage{latexsym}
\usepackage{eufrak}
\usepackage{euscript}

\usepackage[T1]{fontenc}

\usepackage[width=18cm,height=24cm]{geometry}
\usepackage{epsfig}
\usepackage{tikz}
\usepackage{graphics}
\usepackage{array}
\usepackage{enumerate} \usepackage{authblk}

\newtheorem{theorem}{Theorem}[section]
\newtheorem{lemma}[theorem]{Lemma}
\newtheorem{corollary}[theorem]{Corollary}

\newtheorem{remark}[theorem]{Remark}
\newtheorem{example}[theorem]{Example}

\theoremstyle{definition}

\theoremstyle{proof}
\theoremstyle{observation}

\begin{document}
  \setcounter{Maxaffil}{2}
   \title{A short note on cospectral and integral chain graphs for Seidel matrix}
   \author[ ]{Santanu Mandal\thanks{santanu.vumath@gmail.com}}
   %\author[,a]{Ranjit Mehatari\thanks{ranjitmehatari@gmail.com, mehatarir@nitrkl.ac.in}}
   %\author[,b]{ Gary R. W.  Greaves\thanks{gary@ntu.edu.sg}}
   \affil[ ]{Department of Mathematics,}
  \affil[ ]{National Institute of Technology,}
   \affil[ ]{Rourkela - 769008, India}

%\affil[b]{School of Physical and Mathematical Sciences,}
%\affil[ ]{Nanyang Technological University,}
%\affil[ ]{21 Nanyang Link, Singapore 637371, Republic of Singapore}
   \maketitle
\begin{abstract}
In this brief communication, we investigate the cospectral as well integral chain graphs for Seidel matrix, a key component to study the structural properties of equiangular lines in space. We derive a formula that allows to generate an infinite number of inequivalent chain graphs with identical spectrum. In addition, we obtain a family of Seidel integral chain graphs. This contrapositively answers a problem posed by Greaves ["Equiangular line systems and switching classes containing regular graphs", Linear Algebra Appl., (2018)] ("Does every Seidel matrix with precisely three distinct rational eigenvalues contain a regular graph in its switching class?"). Our observation is- "no".
\end{abstract}
\section{Introduction}
Suppose $K_n, C_n$ denote the complete graph and cycle on $n$ vertices respectively. A graph with no induced subgraphs isomorphic to $C_3, C_5$ and $2K_2$ is called a chain graph. Chain graph has several equivalent definitions. In this study, we consider the definition based on binary string, that will be relevant to this work. The underlying matrix is the Seidel matrix $S$.\\
An $n$ vertex chain graph can be generated by the binary string $b=0^{s_1} 1^{t_1} 0^{s_2} \ldots 0^{s_k} 1^{t_k}$, where $s_i,t_i\geq1$ and $\sum s_i+\sum t_i=n$ (see \cite{MMD}). Note that a chain graph is a bipartite graph, and it is complete bipartite if and only if $k=1$.\\
Let $G=(V,E)$ be a simple, connected and finite graph with vertex set $V=\{v_1,v_2,\ldots,v_n\}$. Let $A$ be the $(0,1)$-adjacency  matrix of $G$. Then the Seidel matrix (studied in \cite{Berman, MMD}) is a square matrix of order $n$ defined by $$S=J-I-2A,$$
where $J$ is all $1$ matrix and $I$ is the identity matrix. In other words, if $s_{ij}$ is the $(i,j)$-th entry of $S$, then
$$s_{ij}=\begin{cases}
-1&\text{if }v_i\sim v_j,\\
1&\text{if }v_i\nsim v_j,\ i\neq j,\\
0&\text{if }i=j.
\end{cases}
$$
\subsection{Equiangular line}
A system of lines $\mathcal{L}$ in the Euclidean $d-$dimensional space $\mathbb{R}^d$ is said to be \textbf{equiangular} if all the lines meet at a point and have pairwise equal angle. The constant angle is called the common angle of the system. There is a closed correspondence between the equiangular line and $(-1,0,1)-$Seidel matrix. Seidel and Lint \cite{LS} first notice this important relationship around $1966$. Later Greaves et al. \cite{Greaves,GKMS,GSY} develope several significant results in the theory of equiangular lines in connection with Seidel matrix. In fact, Seidel matrix plays a key role to study the structure of equiangular lines in space geometry. Seidel matrix can be viewed as a Gram matrix corresponding to an equiangular line system of $n>d$ lines in $\mathbb{R}^d$, where the pairwise inner product between any two distinct lines is $\pm\frac{1}{\lambda_{min}}$ (see \cite{LS}), where $\lambda_{min}$ is the smallest eigenvalue of $S$. Suppose $S$ is the Seidel matrix of a graph with $n$ vertices and $\lambda_{min}$ is the least eigenvalue with multiplicity $(n-d)$ for some natural number $d$. Then $\mathcal{L}$  contains at most $n$ equiangular lines in the Euclidean space $\mathbb{R}^d$. The common angle is $\frac{1}{\lambda_{min}}$.  The study of Seidel matrices is motivated by this correspondence, which also emphasises the significance of their smallest eigenvalue. We believe that this work can be of interest to the readers for its tight connection with equiangular lines in the Euclidean spaces (see \cite{Greaves, GKMS, GSY, LS, SO} for more details).

%%%%%%%%%%%%%%%%%%%%%%%%%%%%%%%%%%%%%%%%%%%%%%%%%%%%%%%%%%

\subsection{Switching}
Suppose $G=(V,E)$ is a graph and $V^{'}\subseteq V$. Then the switching graph of $G$ on $V^{'}$ is the graph $G^{'}$ arises from $G$ by changing all edges between $V^{'}$ and $V\setminus V^{'}$ to non-edges, and all the non-edges to edges as well. This operation is called the switching on the subset $V^{'}$. The adjacency property among the vertices in $V$ (or in $V^{'}$) is preserved when switching is applied. That means we leave the edges and non-edges within $V^{'}$ (or $V\setminus V^{'}$) unaltered. The collection of all graphs that can be obtained from $G$ by switching on every possible subset of $V$ is called the switching class of $G$. Two graphs $G$ and $H$ are called the switching equivalent (or simply equivalent) if they belong to the same switching class. Otherwise, we call them simply inequivalent. If two graphs are equivalent (res. inequivalent) then their Seidel matrices are said to be equivalent (res. inequivalent). Equivalent Seidel matrices are cospectral, but starting from $n=8$ examples of inequivalent cospectral Seidel matrices (of order $n$) appear. That means there does not exist any cospectral graphs (which are not switching equivalent) of order $n\leq7$. Sz\"oll\"osi et al. \cite{SO} tabulated the number of inequivalent cospectral Seidel matrices  of order $n$ where $n\leq13$. However, here we construct an infinite family of inequivalent cospectral chain graphs of order $n$ where $n\geq14$. Greaves \cite{Greaves} investigated the case- when a Seidel matrix having precisely three distinct eigenvalues has a regular graph in its switching class. He noted that, excluding one (this contains irrational eigenvalues), every Seidel matrix of order at most $12$ with precisely three distinct eigenvalues has a regular graph in its switching class. And, motivated from this observation, he left a question:\\
\textbf{Question 1.}\cite{Greaves} Does every Seidel matrix with precisely three distinct rational eigenvalues contain a regular graph in its switching class?\\ We are able to answer this question contrapositively by finding one counter example (see Corollary \ref{cor4}). Sz\"oll\"osi and \"Osterg\'ard \cite{SO} pointed out an infinite family of Seidel matrices (irrespective of the nature of the eigenvalues) that do not have a regular graph in their switching classes. Recall that a $\{P_4,C_4,2K_2\}-$ free graph is called a threshold graph. Mandal and Mehatari \cite{MM} proved that - no threshold graphs can have three distinct Seidel eigenvalues.  In light of this conclusion, a straighforward theorem is given below.
\begin{theorem}
No threshold graphs with three distinct Seidel eigenvalues contain a regular graph in its switching class.
\end{theorem}

%%%%%%%%%%%%%%%%%%%%%%%%%%%%%%%%%%%%%%%%%%%%%%%%%%%%%%%%%%

\subsection{Quotient Matrix}
Let us consider a chain graph $G$ with the binary string $b=0^{s_1} 1^{t_1} 0^{s_2} \ldots 0^{s_k} 1^{t_k}$. Then the  Seidel matrix $S$ of $G$ is a square matrix of size $n$, given by

$$S=\begin{bmatrix}
(J-I)_{s_1}&-J_{s_1\times t_1}&J_{s_1\times s_2}&-J_{s_1\times t_2}&J_{s_1\times s_3}&-J_{s_1\times t_3}&\ldots&-J_{s_1\times t_k}\\[1mm]
-J_{t_1\times s_1}&(J-I)_{t_1}&J_{t_1\times s_2}&J_{t_1\times t_2}&J_{t_1\times s_3}&J_{t_1\times t_3}&\ldots&J_{t_1\times t_k}\\[1mm]
J_{s_2\times s_1}&J_{s_2\times t_1}&(J-I)_{s_2}&-J_{s_2\times t_2}&J_{s_2\times s_3}&-J_{s_2\times t_3}&\ldots&-J_{s_2\times t_k}\\[1mm]
-J_{t_2\times s_1}&J_{t_2\times t_1}&-J_{t_2\times s_2}&(J-I)_{t_2}&J_{t_2\times s_3}&J_{t_2\times t_3}&\ldots&J_{t_2\times t_k}\\[1mm]
& & & & & & \ddots \\[1mm]
J_{s_k\times s_1}&J_{s_k\times t_1}&J_{s_k\times s_2}&J_{s_k\times t_2}&J_{s_k\times s_3}&J_{s_k\times t_3}&\ldots &-J_{s_k\times t_k}\\[1mm]
-J_{t_k\times s_1}&J_{t_k\times t_1}&-J_{t_k\times s_2}&J_{t_k\times t_2}&-J_{t_k\times s_3}&J_{t_k\times t_3}&\ldots&(J-I)_{t_k}
\end{bmatrix},$$
where  $J_{m\times n}$ is all 1 block matrix of size $m\times n$ and the diagonal blocks of $S$ are the square matrices of size $s_1\times s_1,~t_1\times t_1,~s_2\times s_2,~t_2\times t_2,~\ldots,~t_k\times t_k$. \\
We now construct an equitable partition in the following way. For the representation $b=0^{s_1} 1^{t_1} 0^{s_2} \ldots 0^{s_k} 1^{t_k}$ of $G$, let $V_{s_1}$ denote the set of vertices representing first $s_1$ added vertices, $V_{t_1}$ denote the set of vertices representing next $t_1$ added vertices, $V_{s_2}$ denote the set of vertices representing next $s_2$ added vertices and finally $V_{t_k}$ denote the set of vertices representing last $t_k$ added vertices. Now we consider the vertex partition of $\pi =\{ C_1,C_2,C_3,\ldots,C_{2k}\}$, where   $C_i=V_{s_j}=s_j$, if $i=2j-1$ $(j=1,\,2,\ldots,\,k)$, \text{and} $C_i=V_{t_j}=t_j$, if $i=2j$ $(j=1,\,2,\ldots,\,k)$. Then $\pi$ is an equitable partition of $G$ of size $2k$. Throughout this paper we consider this equitable partition only, and by writing ``$G$ is a chain graph of order $n$ with $|\pi|=2k$'' we mean that $G$ is a chain graph of order $n$ with the binary string $b=0^{s_1} 1^{t_1} 0^{s_2} \ldots 0^{s_k} 1^{t_k}$. Let $Q_\pi$ be the corresponding quotient matrix. Then  $Q_S$ is a square matrix of size $2k$, given by
%\newpage
$$Q_\pi=\begin{bmatrix}
(s_1 -1) & -t_1 &s_2  &-t_2&s_3  &-t_3& \ldots & -t_k \\[1mm]
-s_1&(t_1 -1) & s_2 &t_2&s_3&t_3& \ldots & t_k\\[1mm]
s_1 &t_1&(s_2-1)&-t_2 &s_3 &-t_3& \ldots & -t_k\\[1mm]
-s_1&t_1&-s_2&(t_2 -1)&s_3&t_3&\ldots&t_k\\[1mm]
s_1&t_1&s_2&t_2&(s_3-1)&-t_3&\ldots&-t_k\\[1mm]
-s_1&t_1&-s_2&t_2&-s_3&(t_3-1)&\ldots& t_k\\[1mm]
& & & & & & \ddots \\[1mm]
s_1&t_1&s_2&t_2&s_3&t_3&\ldots&-t_k \\[1mm]
-s_1 &  t_1 &-s_2 &t_2 &-s_3 &t_3& \ldots & (t_k -1)
\end{bmatrix}.$$
The following results are helpful in this study.
\begin{theorem}\cite{MMD}
\label{Chain_pm1_Ch1} Let $G$ be a chain graph of order $n$ with $|\pi|=2k$. Then
$$n_{-1}(S)= n-2k+1,$$
where $n_{-1}(S)$ is the multiplicity of the eigenvalue $-1$ of $S$.
\end{theorem}
\begin{theorem}\cite{MMD}
\label{Chain_quotient_th2}
An eigenvalue of  $Q_\pi$ can appear at most twice.
\end{theorem}

%%%%%%%%%%%%%%%%%%%%%%%%%%%%%%%%%%%%%%%%%%%%%%%%%%%%%%%%%%%%%%%%%%

\section{Cospectral chain graphs}
Construction of cospectral graphs with respect to Seidel matrix is more difficult than any other graph matrices (like adjacency, Laplacian, distance etc). Because for the most part, the idea of switching-equivalent prevents the occurrence of such graphs in the Seidel matrix scenario. The creation of cospectral graphs with respect to Seidel matrix is a challenging work because of this. We call two graphs Seidel cospectral if they are cospectral on Seidel matrix. Berman et al. \cite{Berman} construct a family of complete tripartite graphs which are Seidel cospectral. Here we produce a family of Seidel cospectral chain graphs. Let us first prove the following lemma which confirms the inequivalency of two chain graphs.
\begin{lemma}
Two chain graphs with the binary strings $B_1=01^{a_1}0^{a_1}1^{b_1}$ and  $B_2=01^{a_2}0^{a_2}1^{b_2}$ are NOT switching equivalent, where $a_1\neq a_2$, but $2a_1+b_1=2a_2+b_2$.
\end{lemma}

\begin{proof}
Let $V_i$ be the independent sets in $B_1$ such that $|V_1|=1, |V_2|=|V_3|=a_1, |V_4|=b_1$. Also let $W_i$ be the independent sets in $B_2$ such that $|W_1|=1, |W_2|=|W_3|=a_2, |W_4|=b_2$. Suppose there exists $U\in \cup_{i=1}^4 V_i$ such that switching $B_1$ with respect to $U$ produces the chain graph $G^{'}=B_2$. Now there arise following cases. \\

\textbf{Case 1.} Let $x\in U$. If $x$ belongs to either of the vertex sets $V_1$, $V_2$, or $V_3$, then switching $B_1$ with respect to $U$ produces no chain graph.\\

\textbf{Case 2.} Let $U=\{x_i\in V_4$ : $i=1,2,3, \cdots, q$\}. Then switching $B_1$ with respect to $U$ produces the chain graph switching-equivalent to $G^{'}=0^{q+1}1^{a_1}0^{a_1}1^{b_1-q}$. \\

\textbf{Case 3.} Let $U=\{x, y : x\in V_1 , y\in V_4\}$. Then switching $B_1$ with respect to $U$ produces the chain graph switching-equivalent to $G^{'}=01^{a_1}0^{a_1}1^{b_1}.$\\

\textbf{Case 4.} Let $U=\{V_i \cup V_j\}$. Now if $(i, j) \in \{(1, 4), (2, 3)\}$, then switching $B_1$ with respect to $U$ produces the chain graph switching-equivalent to $G^{'}=01^{a_1}0^{a_1}1^{b_1}.$ \\

\textbf{Case 5.} Let $U=\{V_i \cup V_j\}$. Now if $(i, j) \in \{(3, 4), (1, 3), (1, 2), (3, 3)\}$, then switching $B_1$ with respect to $U$ produces no chain graph.\\
Therefore from the above cases we conclude that any switching on $B_1$ does not produce $B_2$ if $a_1\neq a_2$. This completes the proof of the lemma.
\end{proof}
Motivated from this lemma we construct a family of Seidel cospectral chain graphs (inequivalent) as follows.
\begin{theorem}
\label{th1}
The chain graphs $ b_1=01^m0^m1^{2m+r}$ and $b_2=01^{2m}0^{2m}1^r$ are Seidel cospectral if there exists an odd positive integer $r$ such that
$$ m=\frac{3}{2}(r+1)$$.
\end{theorem}
\begin{proof}
Note that the given chain graphs are not switching-equivalent as demonstrated in the previous lemma. Their eigenvalues are

\begin{center}
\begin{tabular}{|c|c|}
\hline
Graph & Eigenvalues \\ \hline
$01^m0^m1^{2m+r}$ & $-1^{4m+4r-2}, ~2m-1,~\frac{1}{2}(2m+r-1\pm\sqrt{20m^2+12mr+12m+r^2+2r+1})$\\
\hline
$01^{2m}0^{2m}1^r$ & $-1^{4m+4r-2}, ~4m-1, ~\frac{1}{2}(r-1\pm \sqrt{16mr+16m+r^2+2r+1})$ \\
\hline
\end{tabular}
\end{center}

There arise three cases.\\
\textbf{Case 1.} $4m-1= \frac{1}{2}(2m+r-1+ \sqrt{20m^2+12mr+12m+r^2+2r+1})$. \\ 
\textbf{Case 2.} $2m-1= \frac{1}{2}(r-1+ \sqrt{16mr+16m+r^2+2r+1})$.\\
\textbf{Case 3.} $\frac{1}{2}(2m+r-1- \sqrt{20m^2+12mr+12m+r^2+2r+1})=\frac{1}{2}(r-1- \sqrt{16mr+16m+r^2+2r+1})$.\\
On simplification of the above three cases, we obtain the required result.
\end{proof}
\begin{remark}
This classifies the chain graphs with exactly four distinct Seidel eigenvalues.
\end{remark}
\begin{corollary}
In the above theorem, when cospectrality occurs then the spectral radius is $n-(r+2)$.
\end{corollary}
\begin{corollary}
In Theorem \ref{th1}, when cospectrality occurs then the smallest eigenvalue is simple and given by the formula $\lambda_{min}=-(2m-r)$.
\end{corollary}
As a quick application of this result, here one example is considered to find the number of equiangular lines in the $d-$ dimensional Euclidean space $\mathbb{R}^d$ for some positive integer $d$.
\begin{example}
In the Theorem \ref{th1}, when cospectrality occurs, then the least eigenvalue is $-(2m-r)$ which occurs once. Note that the order of the Seidel matrix is $4m+r+1$.  The corresponding $4m+r+1$ lines form the largest known equiangular line system in $\mathbb{R}^{4m+r}$. And, the common angle is $\frac{1}{2m-r}$.
\end{example}
We now give a list of $10$ pair of Seidel cospectral chain graphs in the following table (see Table \ref{chain_table1}).
\begin{table}[h]
	\begin{center}
		\begin{tabular}{|c|c|c|c|}
			\hline
No.  of vertices & Chain graph$(01^m0^m1^{2m+r})$ & chain graph$(01^{2m}0^{2m}1^r)$ & Eigenvalues \\ \hline
$14$ & $01^30^31^7$ & $01^60^61$ & $-1^{11}, \pm 5, 11$ \\
\hline
$28$ & $01^60^61^{15}$ & $01^{12}0^{12}1^3$ & $-1^{25}, -9, 11, 23 $  \\
\hline
$42$ & $01^90^91^{23}$ & $01^{18}0^{18}1^5$ & $-1^{39}, -13, 17, 35 $  \\
\hline
$56$ & $01^{12}0^{12}1^{31}$ & $01^{24}0^{24}1^7$ & $-1^{53}, -17, 23, 47 $  \\
\hline
$70$ & $01^{15}0^{15}1^{39}$ & $01^{30}0^{30}1^9$ & $-1^{67}, -21, 29, 59 $  \\
\hline
$84$ & $01^{18}0^{18}1^{47}$ & $01^{36}0^{36}1^{11}$ & $-1^{81}, -25, 35, 71 $  \\
\hline
$98$ & $01^{21}0^{21}1^{55}$ & $01^{42}0^{42}1^{13}$ & $-1^{95}, -29, 41, 83 $  \\
\hline
$112$ & $01^{24}0^{24}1^{63}$ & $01^{48}0^{48}1^{15}$ & $-1^{109}, -33, 47, 95 $  \\
\hline
$126$ & $01^{27}0^{27}1^{71}$ & $01^{54}0^{54}1^{17}$ & $-1^{123}, -37, 53, 107 $  \\
\hline
$140$ & $01^{30}0^{30}1^{79}$ & $01^{60}0^{60}1^{19}$ & $-1^{137}, -41, 59, 119 $  \\
\hline
\end{tabular}
\end{center}
\caption{Cospectral chain graphs}
\label{chain_table1}
\end{table}

%%%%%%%%%%%%%%%%%%%%%%%%%%%%%%%%%%%%%%%%%%%%%%%%%%%%%%%%%%%%%%%%%%

\section{Integral chain graphs}
Which graphs have integral spectra? This was first raised by Harary et el. \cite{harary} in $1974$. After that the problem of finding integral graphs become more popular. Integral graphs are actually quite uncommon. We call a graph Seidel integral if the Seidel matrix consists entirely of integer eigenvalue. In this section we try to obtain few Seidel integral chain graphs. Zhao et al. \cite{wei} produced Seidel integral complete multipartite graphs. They obtained a necessary and sufficient condition for a complete tripartite graph $K_{p,q,r}$ to be Seidel integral. Whereas Wang et al. \cite{wang} provided a list of Seidel integral complete r-partite graphs. Suppose $G$ is a chain graph on $n$ vertices with the binary string $b=0^{s_1} 1^{t_1} 0^{s_2} \ldots 0^{s_k} 1^{t_k}$. Here we obtain a family of Seidel integral chain graphs for $k=2$.

\begin{theorem}
\label{th2}
The chain graph with binary string $0^s1^{2s}0^{2s}1^s$ is Seidel integral.
\end{theorem}

\begin{proof}
The Seidel eigenvalues of the given chain graph are
$$-1^{6s-3}, -2s-1, 4s-1, 4s-1,$$
where exponents denote the frequency of occurrence of eigenvalue.
\end{proof}
\begin{remark}
Theorem \ref{th2} classifies the chain graphs with exactly three distinct (rational) Seidel eigenvalues such that its switching class contains a regular graph of valency $3s$. This can be viewed as an example in support of a question posed by Greaves (Question A, \cite{Greaves}) affirmatively.
\end{remark}
\begin{corollary}
For the chain graph $0^s1^{2s}0^{2s}1^s$, $\lambda_{min}=-(2s+1)$ with multiplicity $2$.
\end{corollary}

\begin{theorem}
	\label{chain_integral_2}
The $n$ vertex chain graph with binary string $01^m0^m1^{n-2m-1}$ is Seidel integral if $(n, m)$ takes any one of the following values

$$(n, m) \in \begin{cases}
\{(3r, r)\} &\text{for } r\geq 2,\\
\{(13r, 6r), (13r, 2r)\}&\text{for } r\geq 2,\\
 \{(2r^2+2r+2, r^2+r)\}&\text{for } r\geq 1, \\
\{(4r^2-2r+1, r)\}&\text{for } r\geq 3, \\
\{(4r^2+4r+4, 2r^2+2r)\}& \text{for } r\geq 1,
\end{cases}
$$
where $r$ is a positive integer.
\end{theorem}
\begin{proof}
The Seidel eigenvalues of the given chain graph are
$$-1^{n-3}, ~2m-1,~ -(m+1)+\frac{1}{2}\Big[n \pm \sqrt{(n-2m)(n+6m)}\Big],$$ where exponents denote the frequency of occurrence of eigenvalue. Note that $(n-2m)(n+6m)$ is even (res. odd) if $n$ is even (res. odd). Therefore the eigenvalues are integral if the term $(n-2m)(n+6m)$ is a perfect square.  Now,\\$(n-2m)(n+6m)=(n+2m)^2-(4m)^2=x^2-y^2$, (say), where $x=n+2m$, $y=4m$. Since $n>2m$, so $x>y$. Let $x=y+h$. Therefore
\begin{equation}
\label{equation_19}
n=h+2m
\end{equation}
Now, $x^2-y^2=(y+h)^2-y^2=h(2y+h)$.
It is easy to verify that 
$$ h^2<h(2y+h)<4(h+2m)^2 ,$$
which means that $(n-2m)(n+6m)$ can takes the perfect square value lie in the open interval $(h^2, 4(h+2m)^2)$. The possible perfect square which lie in the open interval $(h^2, 4(h+2m)^2)$ are $(h+m)^2, (h+2m)^2, (h+3m)^2, (h+2r)^2, \{h+m+(2p-1)\}^2~ for~ p\geq 3, (h+i)^2~ for~ i=1, 2, 3, 4, 8$. We now consider several cases to obtain the values of $m, n$ such that the expression within the square root of the eigenvalues become perfect square.\\
\textbf{Case 1.} $h(2y+h)=(h+m)^2$. Simplifying and using equation no. (\ref{equation_19}) we obtain
$$(n, m)=(13r, 6r),$$
where $r$ is an integer $\geq 2$.\\
\textbf{Case 2.} $h(2y+h)=(h+2m)^2$. Simplifying and using equation no.(\ref{equation_19}) we obtain
$$(n, m)=(3r, r),$$
where $r$ is an integer $\geq 2$.\\
\textbf{Case 3.} $h(2y+h)=(h+3m)^2$. Simplifying and using equation no. (\ref{equation_19}) we obtain
$$(n, m)=(13r, 2r),$$
where $r$ is an integer $\geq 2$.\\
\textbf{Case 4.} $h(2y+h)=(h+i)^2~ for~ i=1, 2, 3, 4, 8$. We see that for $i=1, 2, 3$ the value of $h$ becomes fractional, an absurd result. So $i=4, 8$ are the only acceptable case. Simplifying and using equation no. (\ref{equation_19}) we obtain
$$(n, m)\in \{(6, 2), (14, 6), (12, 4)\}.$$
\textbf{Case 5.} $h(2y+h)=(h+2r)^2$. This gives 
\begin{equation}
\label{equation_20}
h=\frac{4r^2}{8m-4r}
\end{equation}
 which should be an integer value. Now the divisors of $4r^2$ are $2, 4, r, r^2, 2r, 2r^2, 4r, 4r^2$. So we take the following subcases.\\
\textbf{Subcase a.} $8m-4r=4$. Simplifying and using equations no. (\ref{equation_19}), (\ref{equation_20}) we obtain
$$(n, m)=(4r^2-2r+1, r),$$ 
where $r$ is an integer $\geq 3$. \\
\textbf{Subcase b.} $8m-4r=4r$. Simplifying and using equations no.  (\ref{equation_19}), (\ref{equation_20}) we obtain 
$$(n, m)=(3r, r),$$ 
where $r$ is an integer $\geq2$. \\
\textbf{Subcase c.} $8m-4r=r^2$. Simplifying and using equations no. (\ref{equation_19}), (\ref{equation_20}) we obtain \\
$$(n, m)=(4r^2+4r+4, 2r^2+2r),$$ 
where $r$ is an integer. \\
\textbf{Subcase d.} $8m-4r=2r^2$. Simplifying and using equations no.  (\ref{equation_19}), (\ref{equation_20}) we obtain 
$$(n, m)=(2r^2+2r+2, r^2+r),$$ 
where $r$ is a positive integer.\\
We discard  all other cases where absurd result found. This completes the proof of the theorem.
\end{proof}
\begin{corollary}
For the chain graph $01^m0^m1^{n-2m-1}$, the minimum eigenvalue is given by  $$\lambda_{min}= -(m+1)+\frac{1}{2}\Big[n - \sqrt{(n-2m)(n+6m)}\Big].$$

\end{corollary}
\begin{remark} 
When cospectrality occurs then Theorem \ref{th1} produces an infinite family of Seidel integral chain graph. Some examples of Seidel integral chain graphs are depicted in Table \ref{chain_table1} and Table \ref{chain_table2}, where exponents denote the repetition of an eigenvalue.
\end{remark}
\begin{table}[h]
	\begin{center}
		\begin{tabular}{|c|c|c|c|}
			\hline
Chain graph$(0^s1^{2s}0^{2s}1^s)$ & Eigenvalues & chain graph $(01^m0^m1^{n-2m-1})$ & Eigenvalues \\ \hline
 $01^20^21$ & $-1^3, -3, 3^2$ & $01^30^31^2$ & $-1^{6},-4,5^2$  \\
\hline
 $0^21^40^41^2$ & $-1^9,-5,7^2$ & $01^30^31^{24}$ & $-1^{28}, -6, 5, 29$  \\
\hline
 $0^31^60^61^3$ & $-1^{15},-7,11^2$ & $01^40^41^{48}$ & $-1^{54}, -8, 7, 55$ \\
\hline
 $0^41^80^81^4$  & $-1^{21},-9,15^2$ & $01^60^61$ & $-1^{11},\pm 5,11$  \\
\hline
 $0^51^{10}0^{10}1^5$ & $-1^{27}, -11,19^2$ & $01^50^51^{80}$ & $-1^{88}, -10, 9, 89$ \\
\hline
 $0^61^{12}0^{12}1^6$ & $-1^{33},-13,23^2$ & $01^40^41^3$ & $-1^9,-5,7^2$   \\
\hline
$0^71^{14}0^{14}1^7$ & $-1^{39},-15,27^2$ & $01^{12}0^{12}1$ & $-1^{23},\pm 7,23$ \\
\hline
 $0^81^{16}0^{16}1^8$ & $-1^{45},-17,31^2$ & $01^40^41^{17}$  & $-1^{23},\pm 7,23$ \\
\hline
 $0^91^{18}0^{18}1^9$ & $-1^{51},-19,35^2$ & $01^50^51^4$ & $-1^{12},-6,9^2$   \\
\hline
 $0^{10}1^{20}0^{20}1^{10}$ & $-1^{57},-21,39^2$ & $01^60^61^{26}$ & $-1^{36},-10,11,35$   \\
\hline
\end{tabular}
\end{center}
\caption{Integral chain graphs}
\label{chain_table2}
\end{table}
\begin{corollary}
\label{cor4}
From Table \ref{chain_table2}, we see that the chain graph $G$ on $15$ vertices with binary string $01^50^51^4$ has three distinct rational eigenvalues: $-1^{12},-6,9^2$ (exponents denote the frequency of occurrence) such that its switching class does not contain a regular graph. We verify this result manually by switching $G$ with respect to all possible subset of $15$ vertices. We notice that no switching produces a regular graph but rather a bi-regular graph (of valencies $7~\& ~8$). This occurs when we apply switching with respect to $7$ vertices ($2$ from $1^5$, $3$ from $0^5$ and $2$ from $1^4$). Therefore, the existence of such a graph negates the Greave's question.
\end{corollary}
%%%%%%%%%%%%%%%%%%%%%%%%%%%%%%%%%%%%%%%%%%%%%%%%%%%%%%%%%%

\section{Acknowledgments}
The  research of Santanu Mandal is supported by the University Grants Commission of India under the beneficiary code BININ01569755. 
%%%%%%%%%%%%%%%%%%%%%%%%%%%%%%%%%%%%%%%%%%%%%%%%%%%%%%%%%%%%%%%%%%
\section{Statements and Declarations} \textbf{Competing Interests:} The author made no mention of any potential conflicts of interest.

\end{document}